\documentclass[12pt,a4paper,psamsfonts]{amsart}

\usepackage{fancyhdr}
\usepackage{appendix}
\usepackage{amssymb,amscd,amsxtra,calc}
\usepackage{mathrsfs}
\usepackage{cmmib57}
\usepackage{multirow}
\usepackage[all]{xy}
\usepackage{longtable}
\usepackage[colorlinks,linkcolor=blue,anchorcolor=blue,citecolor=green]{hyperref}
\usepackage{tikz}

\theoremstyle{plain}
    \newtheorem{thm}{Theorem}[section]

    \newtheorem{corollary}[thm]{Corollary}
    
    \newtheorem{lemma}[thm]{Lemma}
    
    \newtheorem{question}[thm]{Question}
    \newtheorem{theorem}[thm]{Theorem}

\theoremstyle{definition}

    \newtheorem{notation}[thm]{Notation}
    \newtheorem*{notation*}{Notation and Terminology}
      
    \newtheorem{remark}[thm]{Remark}
    
\theoremstyle{remark}

\newcommand{\C}{\mathbb{C}}

\newcommand{\PP}{\mathbb{P}}

\newcommand{\Q}{\mathbb{Q}}

\newcommand{\R}{\mathbb{R}}

\newcommand{\Aut}{\operatorname{Aut}}

\newcommand{\Bir}{\operatorname{Bir}}
\newcommand{\Bim}{\operatorname{Bim}}

\newcommand{\GL}{\operatorname{GL}}

\newcommand{\id}{\operatorname{id}}

\newcommand{\Ker}{\operatorname{Ker}}

\newcommand{\Sing}{\operatorname{Sing}}

\newcommand{\mstriangle}[1]{
\begin{tikzpicture}[x=0.3cm,y=0.3cm]
\draw (-0.4,-0.433) -- (1.4,-0.433);
\draw (-0.2,-0.7794) -- (0.7,0.7794);
\draw (1.2,-0.7794) -- (0.3,0.7794);
\end{tikzpicture}
}
\newcommand{\mssharp}[1]{
\begin{tikzpicture}[x=0.3cm,y=0.3cm]
\draw (-0.8,-0.5) -- (0.8,-0.5);
\draw (-0.8,0.5) -- (0.8,0.5);
\draw (-0.5,-0.8) -- (-0.5,0.8);
\draw (0.5,-0.8) -- (0.5,0.8);
\end{tikzpicture}
}

\makeatletter

\newcommand{\Rmnum}[1]{\expandafter\@slowromancap\romannumeral #1@}
\makeatother

\begin{document}

\title[Jordan property]
{Jordan property for automorphism groups of compact spaces in Fujiki's class $\mathcal{C}$}

\author{Sheng Meng, Fabio Perroni, De-Qi Zhang}

\address{Korea Institute For Advanced Study,
Seoul 02455, Republic of Korea}
\email{ms@u.nus.edu, shengmeng@kias.re.kr}

\address{Universit\`a  degli Studi di Trieste, 34127 Trieste, Italy
}
\email{fperroni@units.it}

\address
{
\textsc{National University of Singapore,
Singapore 119076, Republic of Singapore
}}
\email{matzdq@nus.edu.sg}

\begin{abstract}
Let $X$ be a compact complex space in Fujiki's Class $\mathcal{C}$.
We show that the group $\Aut(X)$ of all biholomorphic automorphisms of $X$ has the Jordan property: there is a (Jordan) constant $J = J(X)$ such that any finite subgroup $G\le \Aut(X)$
has an abelian subgroup $H\le G$ with the index $[G:H]\le J$.
This extends, with a quite different method, the result of Prokhorov and Shramov for Moishezon threefolds.
\end{abstract}



\maketitle
\tableofcontents

\section{Introduction}

We work over the field $\C$ of complex numbers.

It all began with a famous result, proved by Camille Jordan in the year 1878, asserting that for any field $k$ of characteristic $0$ and any positive integer $n$, the (linear) automorphism group $\GL_n(k)$ of an $n$-dimensional vector space has the {\it Jordan property}:
there is a Jordan constant $J = J(n)$ such that every finite subgroup $H \le \GL_n(k)$
has an abelian subgroup $H_1$ of index $[H : H_1] \le J(n)$.

People then wondered whether the same Jordan property is shared by other automorphism groups, for instance, (not necessarily linear) general automorphism groups or even birational automorphism groups of varieties.

More than a century having passed and only recently we achieved a quite well understanding 
of the finite subgroups of the  Cremona group 
of rank $2$, ${\rm Cr}(k)=\Bir(\mathbb{P}_k^2)$.  
In particular, if ${\rm char}(k)=0$, the work of Jean-Pierre Serre  \cite{Ser09} yields (among other results) an  
explicit bound for the order of any finite subgroup of ${\rm Cr}(k)$ 
and it implies that ${\rm Cr}(k)$ has the Jordan property.
Soon after, Popov asked whether the group $\Aut(X)$ (resp. $\Bir(X)$) of all automorphisms (resp. all birational automorphisms) of an algebraic variety $X$ is Jordan
(cf.~\cite[Question 2.30-2.31]{Pop11}).
Popov himself proved that for a projective surface $X$, the group $\Bir(X)$ is Jordan unless
$X$ is birational to the product $\PP^1 \times E$ with $E$ an elliptic curve. Late on Zarhin confirmed that
$\Bir(\PP^1 \times E)$ is not Jordan while $\Aut(X)$ is still Jordan when $X$ is projective and birational to $\PP^1 \times E$ (cf.~\cite[\S 2.2]{Pop11}, \cite[Theorem 5.3]{Ser09}, \cite[Theorem 1.2]{Zar14}, \cite[Theorem 1.3]{Zar15}).

For quasi-projective varieties, Bandman and Zarhin proved that $\Aut(X)$ is Jordan when $\dim(X)=2$ or $X$ is birational to the product $\PP^1 \times A$ with $A$ having no rational curve
(cf.~\cite[Theorem 1.7]{BZ15}, \cite[Theorem 4]{BZ19}).

For algebraic varieties of higher dimensions, with the help of the minimal model program, Prokhorov and Shramov \cite[Theorem 1.8]{PS14} confirmed the Jordan property of the group $\Bir(X)$ for any algebraic variety $X$, assuming either $X$ is non-uniruled,
or $X$ has vanishing irregularity as well as the (then) outstanding
Borisov-Alexeev-Borisov conjecture about the bounded-ness of terminal Fano varieties
which has now been affirmatively confirmed by Birkar \cite[Theorem 1.1]{Bir}.
In particular, the Cremona groups have the Jordan property, confirming a conjecture of Serre.

The first and third authors proved that $\Aut(X)$ is Jordan when $X$ is a projective variety (cf.~\cite{MZ18}).
This result is extended to compact normal K\"ahler spaces by J.~Kim \cite{Kim18}, while Popov offered a much simpler proof by reducing the Jordan property to (real) Lie groups; see \cite[Theorem 5]{Pop18} and also Theorem \ref{thm-kim-popov}.

In non-algebraic cases, compact complex surfaces still behave well. Indeed, Prokhorov and Shramov \cite[Theorems 1.6 and 1.7]{PS} showed that $\Aut(X)$ (and even the group $\Bim(X)$ of all bimeromorphic automorphisms of $X$) are Jordan for any non-projective compact complex surface $X$.

However, one cannot generalize Popov's question further to the settings of non-compact complex manifolds, or diffeomorphism groups of compact Riemmannian manifolds.
We refer to \cite{CPS}, \cite{Pop18} and \cite{Zar19} for the counter examples; see also \cite{BZ17}, \cite{BZ20} and \cite{Mun19} for positive cases.

In higher dimensions, it remains unknown whether the biholomorphic or biregular automorphism group $\Aut(X)$ is Jordan for any {\it compact} complex manifold $X$ or {\it non-projective} algebraic variety $X$, respectively.

We refer to Mundet i Riera \cite[\S 1]{Mun19} for an excellent survey of more related results.

As in \cite{BZ19} and \cite{PS19},
a group $G$ is called {\it strongly Jordan} if $G$ is Jordan and if there is a constant $N = N(G)$ such that any finite abelian subgroup of $G$ is generated by at most $N$ elements.
By \cite[Theorem 2.5]{MS63}, the group $\Aut(X)$ being Jordan automatically implies that it is strongly Jordan for any compact complex manifold $X$ (or any compact complex space by taking an equivariant resolution which is a composition of blowups (cf.~\cite[Theorem 13.2]{BM97}).


Theorem \ref{main-thm} below is our main result. 
The assumption of $X$ being smooth can be weakened to being irreducible  by taking an equivariant resolution.
Recall that a compact reduced complex space is said to be in {\it Fujiki's class $\mathcal{C}$} if it is the meromorphic image of a compact K\"ahler manifold, or equivalently it is bimeromorphic to a compact K\"ahler manifold.
We refer to \cite[Definition 1.1 and Lemma 1.1]{Fuj78}, \cite[Chapter IV, Theorem 5]{Var89} and \cite[Theorem 0.7]{DP04} for equivalent definitions and some properties of Fujiki's class $\mathcal{C}$.

\begin{theorem}\label{main-thm}
Let $X$ be a connected complex manifold and let $Z$ be a non-empty compact complex subspace in Fujiki's class $\mathcal{C}$.
Then the automorphism group $\Aut(X, Z)$ of all biholomorphic automorphisms of $X$ preserving $Z$ is strongly Jordan.
\end{theorem}

The following result is immediately obtained by taking an equivariant resolution to reduce to the smooth case and then applying Theorem 1.1 with $Z=X$.
In particular, it answers the question for the Moishezon manifolds by Prokhorov and Shramov, who proved the case of Moishezon threefolds by using a quite different method (cf.~\cite{PS19}).

\begin{corollary} \label{Cor1}
Let $X$ be a reduced compact complex space.
Then $\Aut(X)$ is strongly Jordan in the following cases (where (1) is a special case of (2)):

\begin{enumerate}
\item $X$ is Moishezon, i.e., $X$ is bimeromorphic to a projective variety.
\item $X$ is in Fujiki's class $\mathcal{C}$, i.e., $X$ is the meromorphic image of a compact K\"ahler manifold.
\end{enumerate}
\end{corollary}

\par \vskip 1pc \noindent
{\bf Difference, with others, of our approach towards non-projective varieties:}

In \cite{PS19}, the authors utilise the maximal rational connected fibration $X \dasharrow V$ which still exists for their Moishezon threefolds (and indeed, for all Moishezon manifolds), and the famous non-uniruled-ness of $V$ due to Graber, Harris and Starr, to either show that $X$ is indeed a projective variety and then apply \cite{MZ18}, or show that $X$ is a rationally connected variety and then apply \cite{PS14}, or reduce to a very general fibre (a complex surface or curve) and then apply \cite{PS}.

Our approach to Theorem \ref{main-thm} is based on a very simple idea: make use of the non-K\"ahler locus of a big $(1,1)$ class $[\alpha]$ on an ``invariant'' subspace $Z$ (in Fujiki's class $\mathcal{C}$)
$$E_{nK}(\alpha):=E_{nK}([\alpha]):=\bigcap\limits_{T\in[\alpha]} \textup{Sing}(T)$$  to find some ``invariant'' K\"ahler submanifold $Z_1 \subseteq Z \subseteq X$. 
Here, the intersection ranges over all K\"ahler currents $T =\alpha+i\partial\bar{\partial}\varphi$ in the class $[\alpha]$, and $\textup{Sing}(T)$ is the complement of the set of points $z\in Z$ such that $\varphi$ is smooth near $z$.
We may need to frequently shrink the automorphism group a bit to keep the ``invariant'' property, while the Jordan property is not affected.
Next, we focus on the (linear) automorphism group of the normal bundle $\mathcal{N}_{Z_1/X}$ as inspired by Mundet i Riera \cite{Mun19}; see Lemma \ref{lem-copy}. This way, we reduce the question on (strongly) Jordan property to the case for compact K\"ahler manifolds by an equivariant compactification of $\mathcal{N}_{Z_1/X}$.
We refer to \cite[\S 2.4]{Tos18} for the further details on the non-K\"ahler locus of a big class.

\vskip 2mm
We end the introduction with the following two questions.

\begin{question}\label{Q2}
Let $X$ be a compact complex manifold.
Suppose $X$ is Moishezon or is in Fujiki's class $\mathcal{C}$.
Is $\Aut_{\tau}(X):=\{g\in \Aut(X)\,|\, g^*|_{H^2(X,\Q)}=\id\}$ a finite-index extension of
the neutral connected component $\Aut_0(X)$ of $\Aut(X)$?
\end{question}

\begin{question}\label{Q1}
Let $X$ be a compact complex manifold.
If $X$ is Moishezon (or in Fujiki's class $\mathcal{C}$), can one find a bimeromorphic model $\widetilde{X}$ of $X$
such that $\Aut(X)$ lifts to $\widetilde{X}$, and $\widetilde{X}$ is projective (or K\"ahler)?
\end{question}

A positive answer to Question \ref{Q1} implies a positive answer to Question \ref{Q2}, by making use
of Fujiki \cite[Theorem 4.8]{Fuj78} or Lieberman \cite[Proposition 2.2]{Lie78} and the norm criterion \cite[Proposition 2.9]{MZ18a} for the pseudo-effective cone and the nef cone.

A positive answer to Question \ref{Q2} will render an alternative proof to Corollary \ref{Cor1} by applying Minkowski Theorem \ref{thm-Mink} to $\GL(H^2(X, \Q))$ in order to reduce to the case for $\Aut_0(X)$ (a Lie group)
which is a known case (cf.~\cite{Pop18}).

Question \ref{Q2} has a positive answer when $X$ is a compact K\"ahler manifold (cf.~\cite{Fuj78}, \cite{Lie78}).

\par \vskip 1pc \noindent
{\bf Acknowledgement.}

The first author would like to thank Doctor Xueyuan Wan for answering several questions on currents.
The third author would like to thank Professor Mihai P\u{a}un for the helpful discussions on Question \ref{Q1}.
The authors would like to thank the referee for valuable suggestions to improve the paper.
The first author is supported by a Research Fellowship of KIAS (MG075501).
The second author is supported by the national projects
PRIN 2015EYPTSB-PE1 ``Geometria delle variet\`a algebriche''
and 2017SSNZAW 005-PE1 ``Moduli Theory and Birational Classification'',
by the research group GNSAGA of INDAM and by FRA 2018 of the University of Trieste.
The third author is supported by an ARF of NUS.

\section{Preliminaries}

We use the following notation throughout this paper.
\begin{notation}\label{notation2.1}
Let $X$ be a connected complex space.

\begin{enumerate}
\item $\Aut(X)$ is the group of all biholomorphic automorphisms of $X$.
\item $\Aut(X,Z):=\{g\in \Aut(X)\,|\, g(Z)=Z\}$ for a subset $Z$ of $X$.
\item $\Aut(\mathcal{E}\to X):=\{g\in \Aut(\mathcal{E})\,|\, g \text{ maps every bundle fibre linearly to some bundle fibre}\}$ for a holomorphic vector bundle $\mathcal{E}$ over $X$.

\par \noindent
In the following, $X$ is further assumed to be smooth and compact.
\item $\Aut_0(X)$ is the neutral connected component of $\Aut(X)$. 
\item $\Aut_{\tau}(X):=\{g\in \Aut(X)\,|\, g^*|_{H^2(X,\Q)}=\id\}$.
Clearly, $\Aut_{\tau}(X) \supseteq \Aut_0(X)$.
\end{enumerate}
\end{notation}

\begin{notation}\label{notation2.2}
Let $X$ be a connected complex manifold and $Z$ a connected complex submanifold.
Let $\mathcal{T}_X$ be the tangent bundle of $X$ and $\mathcal{N}_{Z/X}$ the normal bundle.
Let $g\in \Aut(X, Z)$.
Denote by $\mathcal{T}_g$ the induced tangent automorphism of $\mathcal{T}_X$.
Then we have the following commutative diagram:
\[
\xymatrix{
0\ar[r]&\mathcal{T}_Z\ar[r]\ar[d]^{\mathcal{T}_{g|_Z}}&\mathcal{T}_X|_Z\ar[r]\ar[d]^{{\mathcal{T}_g}|_Z}&\mathcal{T}_X|_Z/\mathcal{T}_Z\ar[r]\ar[d]^{\mathcal{N}_g}&0\\
0\ar[r]&\mathcal{T}_Z\ar[r]&\mathcal{T}_X|_Z\ar[r]&\mathcal{T}_X|_Z/\mathcal{T}_Z\ar[r]&0
}
\]
where $\mathcal{N}_g$ is the induced automorphism of the normal bundle $\mathcal{N}_{Z/X} = (\mathcal{T}_X|_Z)/\mathcal{T}_Z$.
In particular, we have a group homomorphism
$$\mathcal{N}: \Aut(X, Z) \to \Aut(\mathcal{N}_{Z/X}\to Z)$$
via $g\mapsto \mathcal{N}_g$.
\end{notation}

The following result is very important in the proof of Theorem \ref{main-thm}.

\begin{lemma}\label{lem-copy}
Let $X$ be a connected complex manifold and $Z$ a non-empty connected complex submanifold.
Then the kernel of the natural homomorphism 
$$
\mathcal{N}: \Aut(X,Z)\to \Aut(\mathcal{N}_{Z/X}\to Z)
$$
contains no non-trivial subgroup of finite order. In particular,
$\Aut(\mathcal{N}_{Z/X}\to Z)$ contains an isomorphic copy of every finite subgroup of $\Aut(X,Z)$.
\end{lemma}

\begin{proof}
Suppose $g\in \Ker \mathcal{N}$ has finite order.
Then $g|_Z=\id$ (and hence $\mathcal{T}_{g|_Z}=\id$) and $\mathcal{N}_g=\id$.
Let $z\in Z$.
Let $\{x_1,\cdots,x_m, y_1,\cdots, y_n\}$ be a basis of $\mathcal{T}_{X,z}$ such that $\{x_1,\cdots,x_m\}$ is a basis of $\mathcal{T}_{Z,z}$.
Since $\mathcal{T}_{g|_Z}=\id$, we have $\mathcal{T}_{g}|_{z}(x_i)=x_i$ for each $x_i$.
Since $\mathcal{N}_g|_{\mathcal{N}_{Z/X,z}}=\id$, we have $\mathcal{T}_{g}|_{z}(y_j)-y_j\in \mathcal{T}_{Z,z}$ for each $y_j$.
Therefore, under the above basis, $\mathcal{T}_{g}|_{z}$ is a lower triangular matrix with diagonal entries all being $1$.
In particular its eigenvalues are all equal to $1$.
Note that $\mathcal{T}_{g}$ has finite order, hence it is diagonalizable and so it is the identity map. 
This, together with 
$X$ being connected, imply $g=\id$ (cf.~\cite[Lemma 2.1(2)]{Mun19}).
The lemma is proved.
\end{proof}

A group $G$ has {\it bounded finite subgroups} if there is a constant $N = N(G)$ such that any finite subgroup $H\le G$ has order $|H| \le N$.

For the (strongly) Jordan property, we may always replace the group by its normal subgroup with the quotient group having bounded finite subgroups.
Indeed, we have:

\begin{lemma}\label{lem-exact}
Consider the exact sequence of groups
$$1 \to G_1 \to G \to G_2 \to 1.$$
Suppose $G_1$ is Jordan (resp. strongly Jordan) and $G_2$ has bounded finite subgroups.
Then $G$ is Jordan (resp. strongly Jordan).
\end{lemma}

To obtain quotient groups having bounded finite subgroups, we often use the following wonderful theorem of Minkowski which allows us to make use of rational representations of geometric automorphisms; see \cite[Theorem 5, and \S4.3]{Ser07}.

\begin{theorem}[Minkowski]\label{thm-Mink}
$\GL_n(K)$ has bounded finite subgroups, whenever $K$ is a number field. The bound depends only on $n$ and
the field extension degree $[K:\mathbb{Q}]$.
\end{theorem}

There are two ways to obtain the strongly Jordan property for compact K\"ahler manifolds.
We adopt a shorter one here.

\begin{theorem}\textup{(\cite[Theorem 1.1]{Kim18}, \cite[Theorem 2]{Pop18})}\label{thm-kim-popov}
Let $X$ be a compact K\"ahler manifold.
Then $\Aut(X)$ is strongly Jordan.
\end{theorem}

\begin{proof}
The pullback action of $\Aut(X)$ on $H^2(X, \Q)$ gives a faithful rational representation
$$\Aut(X)/\Aut_\tau(X)\hookrightarrow \GL(H^2(X, \Q))$$
with the latter group having bounded finite subgroups by Theorem \ref{thm-Mink}.
Note that $\Aut_\tau(X)/\Aut_0(X)$ is a finite group by \cite[Proposition 2.2]{Lie78} or \cite[Theorem 4.8]{Fuj78}.
Therefore, $\Aut(X)/\Aut_0(X)$ has bounded finite subgroups.
By Lemma \ref{lem-exact},
it suffices to show $\Aut_0(X)$ is strongly Jordan.
By the proof of \cite[Theorem 2]{Pop18},
there exists some $n$ such that $\GL_n(\R)$ contains an isomorphic copy of every finite subgroup of
$\Aut_0(X)$.
Note that $\GL_n(\R)$ is strongly Jordan (cf.~e.g.~\cite[Lemmas 2.3 and 2.4]{MZ18}).
The theorem follows.
\end{proof}

\section{Proof of Theorem \ref{main-thm}, and another open question}

In this section, we prove Theorem \ref{main-thm}, ask Question \ref{Q3} and give Remark \ref{rem_Q3}
illustrating the usefulness of the latter. We begin with the following.

\begin{theorem}\label{thm-jordan-bundle}
Let $X$ be a compact K\"ahler manifold and $\mathcal{E}$ a holomorphic vector bundle of finite rank $r$.
Then $\Aut(\mathcal{E}\to X)$ is strongly Jordan.
\end{theorem}

\begin{proof}
Let $\mathcal{O} = X \times \mathbb{C}$
be the trivial line bundle over $X$.
Let $T:=\mathbb{C}^*$ act on $\mathcal{E} \oplus \mathcal{O}$
by the natural scalar multiplication.
Then there is a natural $T$-equivariant monomorphism
$$\phi:\Aut(\mathcal{E}\to X)\to \Aut(\mathcal{E}\oplus \mathcal{O} \to X)$$
via $g\mapsto g\oplus (g|_X \, \times \, \id_{\C})$.
Hence we have the following $\Aut(\mathcal{E}\to X)$-equivariant commutative diagram
\[
\xymatrix{
\mathcal{E}\ar@{^{(}->}[r]\ar[d] & \mathcal{E}\oplus \mathcal{O}\ar@{.>}[r]\ar[d] & \mathbb{P}(\mathcal{E}\oplus \mathcal{O})\ar[d]\\
X\ar@{=}[r]&X\ar@{=}[r]&X
};
\]
here  $\mathbb{P}(\mathcal{E}\oplus \mathcal{O}):=(\mathcal{E}\oplus \mathcal{O})/T$ is an analytic $\mathbb{P}^{r}$-bundle over $X$; hence it is again a compact K\"ahler manifold (cf.~\cite[Proposition 3.18, page78]{Voi07}).

Consider the homomorphism
$$\psi:\Aut(\mathcal{E}\oplus \mathcal{O} \to X)\to \Aut(\mathbb{P}(\mathcal{E}\oplus \mathcal{O}))$$
induced by the $T$-quotient.
Note that $\psi\circ\phi$ is still a monomorphism.
By Theorem \ref{thm-kim-popov}, $\Aut(\mathbb{P}(\mathcal{E}\oplus \mathcal{O}))$ and hence $\Aut(\mathcal{E}\to X)$ are strongly Jordan.
\end{proof}

We wonder whether we can drop the assumption of $X$
being compact K\"ahler in Theorem \ref{thm-jordan-bundle}.

\begin{question}\label{Q3}
Let $X$ be a connected complex manifold and $\mathcal{E}$ a holomorphic vector bundle of finite rank.
Will $\Aut(\mathcal{E}\to X)$ or $\Aut(\mathbb{P}(\mathcal{E}\oplus \mathcal{O}))$ be (strongly) Jordan if so is $\Aut(X)$?
\end{question}

We can make use of Lemma \ref{lem-copy} and Theorem \ref{thm-jordan-bundle} to prove the following:

\begin{theorem}\label{thm-jordan-Z}
Let $X$ be a connected complex manifold and $Z\subseteq X$ a non-empty connected compact K\"ahler submanifold.
Then $\Aut(X,Z)$ is strongly Jordan.
\end{theorem}

\begin{proof}
By Lemma \ref{lem-copy}, $\Aut(\mathcal{N}_{Z/X}\to Z)$ contains an isomorphic copy of every finite subgroup of $\Aut(X,Z)$.
By Theorem \ref{thm-jordan-bundle}, $\Aut(\mathcal{N}_{Z/X}\to Z)$ is strongly Jordan.
The theorem follows.
\end{proof}

\begin{remark}\label{rem_Q3}
Together with \cite[Theorem 1.6]{PS}, a positive answer to Question \ref{Q3} for compact complex surfaces will deduce the (strongly) Jordan property of $\Aut(V)$ for every compact complex threefold $V$ with $\Sing(V)\neq \emptyset$.
Indeed, just take an equivariant log resolution and apply the same proof of Theorem \ref{thm-jordan-Z} for $X$ being any (smooth) exceptional prime divisor which is a compact complex surface.
\end{remark}

Note that the $\partial\bar{\partial}$-lemma holds for compact complex manifolds in Fujiki's class $\mathcal{C}$.
So it is free for us to use the equivalent Bott-Chern ($\partial\bar{\partial}$),  Dolbeault ($\bar{\partial}$) and De Rham ($d$) cohomologies.
Moreover, Hodge decomposition holds true.
We refer to \cite[Lemma 5.15 and Proposition 5.17]{DGMS75} and \cite[Proposition 1.6 and Corollary 1.7]{Fuj78} for the details.

Now we are ready for:

\begin{proof}[Proof of Theorem \ref{main-thm}]
We take the reduced structure of $Z$.
We first let $G:=\Aut(X,Z)$.
To show $G$ is strongly Jordan,
we shrink $Z$ and $G$ by running the Main Program several times such that the final $Z$ is a $G$-invariant non-empty connected compact K\"ahler submanifold of $X$.
Note that Fujiki's class $\mathcal{C}$ is closed under taking closed subspaces and compact meromorphic images.

\vskip 1pc \noindent
\textbf{Main Program.}

If $Z$ is smooth connected compact K\"ahler, then we stop.

If $Z$ is not connected, we run Step A and restart the Main Program; else we continue.

If $Z$ is singular, we run Step B and restart the Main Program; else we continue.

Else: assume that $Z$ is smooth connected compact in $\mathcal{C}$ but not K\"ahler.
Let $$G_{\tau}:=\{g\in G\,|\, (g|_Z)^*|_{H^2(Z,\Q)}=\id\}.$$
By
the Hodge decomposition
(which still exists for those in
$\mathcal{C}$), $G_{\tau}$ acts trivially, via pullback, on $H^{1,1}_{\partial\bar{\partial}}(Z,\R)$.
Since $Z\in \mathcal{C}$, there is a big real $(1,1)$-class $[\alpha]\in H^{1,1}_{\partial\bar{\partial}}(Z,\R)$ by \cite[Theorem 0.7]{DP04}.
Note that its non-K\"ahler locus $E_{nK}(\alpha)$ is a $G_{\tau}$-invariant non-empty closed analytic subset of $Z$ with $\dim(E_{nK}(\alpha))<\dim(Z)$ (cf.~\cite[\S2.4]{Tos18}).
We replace $Z$ by $E_{nK}(\alpha)$ (with reduced structure) and $G$ by $\Aut(X, E_{nK}(\alpha))$.
Since $G_{\tau} \subseteq \Aut(X, E_{nK}(\alpha))$, if $\Aut(X, E_{nK}(\alpha))$ is strongly Jordan, so is $G_\tau$
and, by noting that $G/G_{\tau}$ has bounded finite subgroups (cf.~Lemma \ref{lem-exact} and Theorem \ref{thm-Mink}), also $G$ is strongly Jordan.
Then we restart the Main Program.

\vskip 1pc \noindent
\textbf{Step A.}
We replace $G$ by its 
subgroup of finite index such that $G$ fixes all (finitely many) connected components of $Z$.
We replace $Z$ by one of its connected component (with reduced structure).

\vskip 1pc \noindent
\textbf{Step B.}
If $Z$ is singular, then its singular locus $\Sing(Z)$ is a $G$-invariant non-empty closed analytic subset of $Z$ with $\dim(\Sing(Z))<\dim(Z)$.
We then replace only $Z$ by $\Sing(Z)$ (with reduced structure).

\vskip 1pc \noindent
\textbf{End of the proof:}

Note that our finitely many replacements of $G$ always fit the assumption in Lemma \ref{lem-exact}.
So it suffices to show that the finally chosen $G$ is still strongly Jordan.
By Theorem \ref{thm-jordan-Z}, $\Aut(X,Z)$ is strongly Jordan, because our finally chosen $Z$ is non-empty, smooth, and compact K\"ahler.
Since our current $G$ is contained in $\Aut(X, Z)$, Theorem \ref{main-thm} follows.
\end{proof}

\end{document}